\documentclass[11pt]{amsart}

\usepackage{amsmath,amssymb,amsthm}
\usepackage[latin9]{inputenc}
\usepackage{geometry}
\usepackage[all]{xy}
\usepackage{graphicx}
\usepackage{enumerate}
\usepackage[hypertexnames=false]{hyperref}
\usepackage{refstyle}
\usepackage{xfrac}

\makeatletter
\def\l@subsection{\@tocline{2}{0pt}{2.5pc}{5pc}{}}
\makeatother

\numberwithin{equation}{section}
\numberwithin{figure}{section}

\theoremstyle{plain}
\newtheorem{thm}{Theorem}[subsection]
\newtheorem{prop}[thm]{Proposition}
\newtheorem{lem}[thm]{Lemma}

\theoremstyle{definition}
\newtheorem{defn}[thm]{Definition}

\theoremstyle{remark}
\newtheorem{example}[thm]{Example}

\newref{sub}{name=§}
\newref{sec}{name=§}
\newref{prop}{name=Proposition~}
\newref{thm}{name=Theorem~}
\newref{lem}{name=Lemma~}
\newref{exam}{name=Example~}

\global\long\def\ma#1{\mathrm{#1}}
\global\long\def\inner#1{\langle#1\rangle}

\global\long\def\l{\lambda}

\global\long\def\R{\mathbb{R}}

\global\long\def\Z{\mathbb{Z}}

\global\long\def\N{\mathbb{N}}

\global\long\def\cO{\mathcal{O}}

\global\long\def\cL{\mathcal{L}}

\global\long\def\t{\mathfrak{t}}

\global\long\def\ra{\rightarrow}

\global\long\def\id{\ma{id}}
\global\long\def\xymt{\ar@{|->}}
\global\long\def\xyinj{\ar@{^{(}->}}
\global\long\def\xydash{\ar@{-->}}

\global\long\def\ket#1{|#1\rangle}

\global\long\def\p#1{\frac{\partial}{\partial#1}}
\global\long\def\ber{\mathrm{Ber}}
\global\long\def\sgn{\mathrm{sgn}}
\global\long\def\pff{\mathrm{pfaff}}
\global\long\def\abs#1{\lvert#1\rvert}

\global\long\def\red{\mathrm{red}}
\global\long\def\norm#1{\lVert#1\rVert}
\global\long\def\hess{\mathrm{Hess}}
\global\long\def\ori{\mathrm{or}}

\begin{document}

\title{Localization and Stationary Phase Approximation on Supermanifolds}
\begin{abstract}
Given an odd vector field $Q$ on a supermanifold $M$ and a $Q$-invariant
density $\mu$ on $M$, under certain compactness conditions on $Q$,
the value of the integral $\int_{M}\mu$ is determined by the value
of $\mu$ on any neighborhood of the vanishing locus $N$ of $Q$.
We present a formula for the integral in the case where $N$ is a
subsupermanifold which is appropriately non-degenerate with respect
to $Q$. 

In the process, we discuss the linear algebra necessary to express
our result in a coordinate independent way. We also extend stationary
phase approximation and the Morse-Bott Lemma to supermanifolds.
\end{abstract}

\author{Valentin Zakharevich}

\maketitle
\tableofcontents{}

\section{Introduction}

The localization phenomenon in the form discussed in this paper, although
not in the language of super geometry, first appeared in the work
of Duistermaat and Heckman \cite{Duistermaat_Heckman}. They compute
the volume of a symplectic manifold with a Hamiltonian circle action
in terms of data at the fixed points of the action. The ideas were
extended to localization in equivariant cohomology by Berline and
Vergne \cite{Berline_Vergne} and independently by Atiyah and Bott
\cite{Atiyah_Bott_moment_map}. 

The study of supermanifolds originated in physics in order to understand
supersymmetry and in particular to make supersymmetry manifest on
the level of classical physics. The path integral quantization in
this case involves an integral over an infinite-dimensional supermanifold.
Witten \cite{Witten_Supersymmetry_and_Morse_Theory} was the first
to apply the localization techniques to supersymmetric theories in
order to reduce infinite dimensional path integrals to finite dimensional
ones. Localization turned out to be a very powerful tool for studying
non-perturbative aspects of supersymmetric field theories. For a survey
of recent applications of localization techniques to supersymmetric
quantum field theories see \cite{Pestun_Localization_techniques}.

In order to isolate the essential features of localization, Schwarz
and Zaboronsky \cite{SuppersymmetryAndLocalization} analyzed  the
case of general finite dimensional supermanifolds. They show that
given an odd vector field $Q$ on a supermanifold $M$ such that $[Q,Q]$
comes from an action of a compact torus and a compactly supported
$Q$-invariant density $\mu$, the integral $\int_{M}\mu$ depends
only on the value of $\mu$ on any open neighborhood of the vanishing
space $N$ of $Q$.%
\footnote{It will be made precise what we mean by the vanishing space in \subref{General-Localization-Statement}.%
} Moreover, in the case where $N$ is an appropriately non-degenerate
discrete subsupermanifold with respect to $Q$, Schwarz and Zaboronsky
derive a formula for the integral $\int_{M}\mu$ in terms of the restriction
to $N$ of $\mu$ and the action of the Lie derivative $\cL_{Q}$
on the normal bundle of $N$. In this paper we extend their result
to the case where $N$ is a non-degenerate subsupermanifold which
need not be discrete (\thmref{localization_formula}). 

We summarize the arguments of Schwarz and Zaboronsky from \cite{SuppersymmetryAndLocalization}.
Under the assumptions that $Q^{2}$ comes from an action of a compact
torus, the authors construct an odd function $\sigma$ having the
properties that $Q^{2}\sigma=0$ and $Q\sigma$ is invertible on the
complement of $N$. Then, if the support of $\mu$ is disjoint from
$N$ we have 
\[
\int_{M}\mu=\int_{M}\cL_{Q}(\frac{\sigma}{Q\sigma}\mu)=0
\]
where the last equality follows from the invariance of integration
under diffeomorphisms. More generally, as we will review in \secref{Localization-Theorem},
given an open neighborhood $U\supset N$ there is a $Q$-invariant
even function $g$ which equals 1 in a neighborhood of $N$ and vanishes
outside of $U$. By the previous argument, for any such $g$ we have
\[
\int_{M}\mu=\int_{M}g\mu+\int_{M}(1-g)\mu=\int_{M}g\mu.
\]
To compute the actual value of the integral, the authors consider
the function
\[
Z(\lambda):=\int_{M}\mu e^{i\lambda Q\sigma}
\]
for $\lambda\in\mathbb{R}_{\geq0}$. Since 
\[
\frac{d}{d\lambda}Z(\lambda)=i\int_{M}\mu Q\sigma e^{i\lambda Q\sigma}=i\int_{M}\mathcal{L}_{Q}(\mu\sigma e^{i\lambda Q\sigma})=0,
\]
the function $Z(\lambda)$ is constant. The stationary phase approximation
computes the asymptotic behavior of $Z(\lambda)$ as $\lambda\ra\infty$
in terms of the local data of $\mu$ and $Q\sigma$ on the critical
subsupermanifold of $Q\sigma$ in case the critical subsupermanifold
is non-degenerate. Since $Z(\lambda)$ does not depend on $\lambda$,
the limit ${\displaystyle \lim_{\lambda\ra\infty}Z(\lambda)}$ equals
the desired integral ${\displaystyle \int_{M}\mu}=Z(0)$. 

In this paper, we first work out an expression for stationary phase
approximation on supermanifolds, i.e. assuming $Q\sigma$ has non-degenerate
critical subsupermanifold, we compute the asymptotic behavior of $Z(\lambda)$
in terms of the Hessian of $Q\sigma$ and restriction of $\mu$ to
the critical subsupermanifold (\thmref{super_stationary_phase_approx}).
To do this, we prove a generalization of the Morse-Bott Lemma to supermanifolds
(\thmref{super_morse_bott1}). We then show that if the vanishing
space of $Q$ is a non-degenerate%
\footnote{The vanishing locus $N$ of a vector field $Q$ is non-degenerate
if it is a subsupermanifold and the restriction of the Lie derivative
$\cL_{Q}$ to the normal bundle $\nu_{N}$ is an automorphism. %
} subsupermanifold $N$, then $N$ is also the non-degenerate critical
subsupermanifold of $Q\sigma$ and we compute the Hessian of $Q\sigma$
in terms of the Lie derivative $\cL_{Q}$. In particular, we end up
with a formula (\thmref{localization_formula}) for $\int_{M}\mu$
in terms of $Q$ and $\mu$ not depending on the auxiliary function
$\sigma$. This has been done in \cite{SuppersymmetryAndLocalization}
in the case where $N$ consists of isolated points. 

In \secref{Super-Linear-Algebra} and \secref{Calculus-on-Supermanifolds},
we will briefly review constructions of super linear algebra and supermanifolds
that will be necessary in later chapters. For a more elaborate introduction,
see \cite{NotesOnSuperSymmetry,Leites1980,GaugeFieldTheory}. In \secref{Stationary-Phase-Approximation}
we will discuss the Morse-Bott Lemma and stationary phase approximation
in the setting of supermanifolds. In \secref{Localization-Theorem},
we state and prove the localization theorem.

\section{Super Linear Algebra\label{sec:Super-Linear-Algebra}}

\subsection{Super Vector Spaces and Super Algebras}

A super vector space is a vector space $V$ together with a decomposition
\[
V=\sum_{i\in\mathbb{Z}/2\mathbb{Z}}V_{i}=V_{0}\oplus V_{1}.
\]
We call elements in $V_{i}$ homogeneous and denote the parity of
homogeneous elements by $p$, i.e., $p(v)=i$ if $v\in V_{i}$. 

The space of linear maps between super vector spaces is naturally
a super vector space, i.e., 
\[
\mbox{Hom}(V,W)_{i}=\{f\in\mbox{Hom}(V,W)|\forall v\in V_{0}\cup V_{1},\ p(f(v))=(-1)^{i}p(v)\}
\]
where in the above expression it is implicit that homogeneous elements
of $\mbox{Hom}(V,W)$ map homogeneous elements of $V$ to homogeneous
elements of $W$. We will often assume that elements in our expressions
are homogeneous in which case the general formula is implied by linearity. 

The direct sum of super vector spaces is defined by 
\[
\left(V\oplus W\right)_{i}=V_{i}\oplus W_{i}.
\]
The tensor product is given by 
\[
\left(V\otimes W\right)_{i}=\bigoplus_{k+l=i}V_{k}\otimes W_{l}.
\]

Of great importance is the isomorphism between $V\otimes W$ and $W\otimes V$
that makes the category of super vector spaces into a symmetric monoidal
category. We define this isomorphism by
\[
v\otimes w\mapsto(-1)^{p(v)p(w)}w\otimes v
\]
where $v\in V,w\in W$ are homogeneous elements. The introduction
of the minus sign when two odd elements are permuted is called the
Koszul sign rule.

A super algebra is a super vector space $A$ together with an even
homomorphism 
\[
A\otimes A\rightarrow A.
\]
We call a super algebra commutative if for $a,b\in A$ homogeneous,
we have 
\[
ab=(-1)^{p(a)p(b)}ba.
\]
From here onwards, by super algebra we will mean an associative super
algebra with a unit.
\begin{example}
Let $A=\mathbb{R}[\theta^{1},\dots,\theta^{n}]$ be the free commutative
super algebra over $\mathbb{R}$ generated by odd elements $\theta^{1},\dots,\theta^{n}$.
A general element of $A$ has the form 
\[
f=\sum_{I\subset\{1,\dots,n\}}a_{I}\theta^{I}
\]
where $\theta_{I}:=\prod_{i\in I}\theta^{i}$ with the product taken
with increasing order of index and $a_{I}\in\mathbb{R}$. Note that
$(\theta^{i})^{2}=0$ for all $i$ and $\theta^{i}\theta^{j}=-\theta^{j}\theta^{i}$. 
\end{example}

\subsection{Modules}

Let $A$ be a super algebra. A left $A$-module $M$ is a super vector
space with the action given by an even morphism of super vector spaces
\[
A\otimes M\rightarrow M
\]
satisfying 
\begin{eqnarray*}
1.m & = & m\hspace{1em}\forall m\in M\\
a.(b.m) & = & (ab).m\hspace{1em}\forall a,b\in A,\hspace{1em}\forall m\in M.
\end{eqnarray*}

Right $A$-modules are defined analogously. If $A$ is commutative
then a left module structure on $M$ defines a right module structure
via 
\[
m.a=(-1)^{p(a)p(m)}a.m
\]

We define the parity reversal functor $\Pi$ on left $A$-modules
in the following way: Let $\tilde{\Pi}$ be a free left $A$-module
generated by an odd element $\pi$. We give it a right action of $A$
by 
\[
(a.\pi).b=(-1)^{p(b)}ab.\pi
\]
For any left $A$-module $M$ we define 
\[
\Pi M:=\tilde{\Pi}\otimes_{A}M.
\]
As an ungraded vector space, $\Pi M$ is isomorphic to $M$ but the
parity of homogeneous elements is flipped. We use the notation 
\[
\Pi m:=\pi\otimes m
\]
where $m\in M$. The $A$ action on $\Pi M$ is 
\[
a.(\Pi m)=(-1)^{p(a)}\Pi(a.m).
\]

A homomorphism between two right $A$-modules $N,M$ is a map 
\[
f:M\rightarrow N
\]
such that 
\[
f(m.a)=f(m).a
\]
A homomorphism $f$ between two left $A$-modules $M,N$ is such that
\[
f(am)=(-1)^{p(a)p(f)}af(m).
\]
One can check that when $A$ is commutative, the notions of morphisms
between $M$ and $N$ as left and right $A$-modules coincide. The
vector space $\mbox{Hom}_{A}(M,N)$ is naturally graded by whether
a homomorphism preserves the grading or reverses it. 

Assume from now on that $A$ is a commutative super algebra. A free
module $M$ over $A$ has dimension $p|q$ if it is freely generated
by $p$ even elements and $q$ odd elements. Let $M$ be a free module
and consider a basis $\{e_{1},\dots,e_{p},e_{p+1},\dots,e_{p+q}\}$
of $M$ where $e_{1},\dots,e_{p}$ are even and $e_{p+1},\dots,e_{p+q}$
are odd. An element $m\in M$ can be given by right coordinates $\{m_{i}\}$:
\[
m=\sum_{i}e_{i}m_{i}.
\]

We will briefly develop the conventions we will use when describing
operations on modules via matrices in a fixed bases. We will denote
column vectors by $\ket{\cdot}$, i.e. for $m\in M$, $\ket m$ is
the column vector of right coordinates of $m$. Let $M,N$ be free
modules with bases $\beta=\{e_{i}\}$ and $\gamma=\{f_{j}\}$ respectively.
Given a morphism $F:M\rightarrow N$, define the matrix $F_{\beta}^{\gamma}$
by 
\[
F(e_{i})=\sum_{j}f_{j}\cdot[F_{\beta}^{\gamma}]{}_{ij}.
\]
We will omit explicit reference to the bases when there is no ambiguity.
It is straightforward to check that 
\[
\ket{Fm}=F\ket m.
\]
The matrix of $F$ is naturally written in block form. If $M$ and
$N$ have dimensions $p|q$ and $p'|q'$ respectively then 
\begin{eqnarray}
F & = & \left(\begin{matrix}A & B\\
C & D
\end{matrix}\right)\label{eq:BlockMatrix}
\end{eqnarray}
where $A\in\mbox{Mat}(p'\times p)$, $B\in\mbox{Mat}(q'\times p)$,
$C\in\mbox{Mat}(p'\times q)$ and $D\in\mbox{Mat}(q'\times q)$. If
$F$ is an even homomorphism, then entries of $A$ and $D$ are even
while entries of $B$ and $C$ are odd. 

If $\beta=\{e_{i}\}$ and $\beta'=\{e_{i}'\}$ are two bases, the
coordinate change matrix from basis $\beta$ to basis $\beta'$ is
the matrix $G_{\beta}^{\beta'}$ such that 
\[
e_{i}=\sum_{j}e'_{j}\cdot[G_{\beta}^{\beta'}]{}_{ij}.
\]
If $\ket m$ is the vector representation of $m\in M$ in the basis
$\beta$, then $G_{\beta}^{\beta'}\ket m$ is the vector representation
of $m$ in the basis $\beta'$. 

For $M$ an $A$-module, the dual module $M^{*}$ is defined by 
\[
M^{*}:=\mbox{Hom}_{A}(M,A).
\]
A basis $\beta=\{e_{i}\}$ of $M$ defines a dual basis $\beta^{*}=\{e_{i}^{*}\}$
of $M^{*}$ by 
\[
e_{i}^{*}(e_{j})=\delta_{i}^{j}.
\]

A morphism $F:M\rightarrow N$ induces a morphism 
\[
F^{*}:N^{*}\rightarrow M^{*}
\]
by the formula 
\[
F^{*}(n^{*})(m):=(-1)^{p(n^{*})p(F)}n^{*}\left(F(m)\right)
\]
for all $n^{*}\in N^{*},m\in M$ homogeneous. To state what the matrix
representation of the dual homomorphism is, we need to introduce the
super analog of the transpose of a matrix. 

For a matrix 
\[
F=\left(\begin{matrix}A & B\\
C & D
\end{matrix}\right)
\]
corresponding to a homogeneous morphism as in \eqref{BlockMatrix}
we define the supertranspose by 
\[
F{}^{st}=\begin{cases}
\left(\begin{matrix}A^{t} & C^{t}\\
-B^{t} & D^{t}
\end{matrix}\right) & p(F)=0\\
\left(\begin{matrix}A^{t} & -C^{t}\\
B^{t} & D^{t}
\end{matrix}\right) & p(F)=1.
\end{cases}
\]
To be precise, the supertranspose is defined on a matrix consisting
of elements of $A$ only after we know the dimensions of the free
modules it is meant to act on. For free modules $M,N$ with fixed
bases and a homomorphism $F:M\rightarrow N$, the matrix of the dual
homomorphism $F^{*}$ in the dual bases is given by the supertranspose,
i.e.,
\[
(F^{*})_{\beta'^{*}}^{\beta^{*}}=\left(F_{\beta}^{\beta'}\right){}^{st}.
\]

\subsection{The Berezinian\label{sub:The-Berezinian}}

The Berezinian is the super analog of the determinant. It is defined
on even automorphisms of free $A$-modules. Let $F$ be an even automorphism
of a free module $M$. Given a basis of $M$, the matrix representation
of $F$ can be written in block form as 
\[
F=\left(\begin{matrix}A & B\\
C & D
\end{matrix}\right).
\]
The Berezinian of $F$ is defined by 
\[
\ber(F)=\det(A-BD^{-1}C)\det(D)^{-1}.
\]
The expression for the Berezinian does not depend on the chosen basis
\cite[§ 1.10]{NotesOnSuperSymmetry}.

Classically, one can define determinant of an automorphism by its
action on the determinant line which is the top exterior power. In
the same spirit, we can define the Berezinian line of a module by
prescribing how an automorphism of the module acts on this line. Let
$M$ be a free $A$-module of dimension $p|q$. The Berezinian line
of $M$, denoted by $\ber M$ has dimension $1|0$ if $q$ is even
and $0|1$ if $q$ is odd. A basis $\beta=\{e_{1},\dots,e_{p+q}\}$
of $M$ defines a basis element $b_{\beta}$ of $\ber M$ and for
an even automorphism $G$ of $M$, the elements $b_{\beta}$ and $b_{G\beta}$
satisfy the relation 
\[
b_{G\beta}=\ber(G)b_{\beta}.
\]
An even isomorphism $F:N\rightarrow M$ induces a map $\ber F:\ber N\rightarrow\ber M$
by 
\begin{eqnarray*}
b_{\beta} & \mapsto & b_{F\beta}.
\end{eqnarray*}

One can check that for an even automorphism $F:M\rightarrow M$, 
\begin{eqnarray*}
\ber(F) & = & \ber(F^{*}).
\end{eqnarray*}
Also, if we denote by $\Pi F:\Pi M\rightarrow\Pi M$ the automorphism
of $\Pi M$ induced by $F$, then 
\[
\ber(\Pi F)=\ber(F)^{-1}.
\]
It follows that there are natural isomorphisms between $\ber(M)^{*}$,
$\ber(M^{*})$ and $\ber(\Pi M)$. For a basis $\beta$, these isomorphisms
identify $\left(b_{\beta}\right)^{*}$, $b_{\beta^{*}}$ and $b_{\Pi\beta}$
where $\Pi\beta=\{\Pi e_{p+1},\dots,\Pi e_{p+q},\Pi e_{1},\dots,\Pi e_{p}\}$
if $\beta=\{e_{1},\dots,e_{p},e_{p+1},\dots,e_{p+q}\}$.

\subsection{Orientation}

Let $A$ be a commutative super algebra over $\R$.
\begin{defn}
A sign homomorphism is a group homomorphism 
\[
\sgn:A^{\times}\rightarrow\{\pm1\}
\]
extending the usual sign on $\R^{\times}$ where $A^{\times}$ denotes
the invertible elements of $A$. 
\end{defn}
We refer to elements in $\sgn^{-1}(+1)$ as positive and those in
$\sgn^{-1}(-1)$ as negative. Let $M$ be a free $p|q$ module over
$A$. There are multiple, non-equivalent notions of orientation on
$M$ coming from even and odd parts of $M$.
\begin{defn}
For $D=\left(\begin{matrix}D_{11} & D_{12}\\
D_{21} & D_{22}
\end{matrix}\right)\in GL(p|q,A)$ , we define the $(i,j)$ orientation of $D$ by 
\[
\ori{}_{(i,j)}(D)=\sgn(\det(D_{11})^{i}\det(D_{22})^{j})
\]
where $i,j\in\{0,1\}$.
\end{defn}
It is straight forward to show that $\ori{}_{(i,j)}$ is well defined
on $GL(M)$ i.e., does not depend on the chosen basis.
\begin{defn}
For $i,j\in\{0,1\}$, we define $\ori_{(i,j)}(M)$ to be the set of
equivalence classes of bases of $M$ where two bases $\beta,\beta'$
are equivalent if $\ori_{(i,j)}$ of the matrix corresponding to the
change of basis from $\beta$ to $\beta'$ is $1$. 
\end{defn}

\subsection{Bilinear Forms\label{sub:Bilinear-Forms}}

Let $A$ be a commutative super algebra over $\R$ with a sign homomorphism
and $M$ a free $A$-module. A bilinear form on $M$ is a morphism
$B:M\otimes M\rightarrow A$. A bilinear form $B$ induces a map 
\[
\hat{B}:M\rightarrow M^{*}
\]
by 
\[
\hat{B}(m_{1})m_{2}=B(m_{1},m_{2}).
\]
If $B$ is a non-degenerate even bilinear form, then $\hat{B}$ is
an even isomorphism and thus induces a map
\[
\ber(\hat{B}):\ber(M)\xrightarrow{\sim}\ber(M^{*})\cong\ber(M)^{*}.
\]

\begin{defn}
For $B$ a non-degenerate even bilinear form, we define $\ber(B)\in\ber(M^{*})^{\otimes2}$
to be the image of $1$ under the map 
\[
A\cong\ber(M)\otimes\ber(M)^{*}\xrightarrow{(\ber(\hat{B})\otimes\id)}\ber(M^{*})^{\otimes2}.
\]
We also define the element $\ber^{-1}(B)\in\ber(M)^{\otimes2}$, which
is the unique element satisfying
\[
\ber(B)\left(\ber^{-1}(B)\right)=1
\]
where we used the fact that for any module $N$, 
\[
N^{*}\otimes N^{*}\cong\left(N\otimes N\right)^{*}
\]
where the isomorphism is given by 
\[
(n_{1}^{*}\otimes n_{2}^{*})(n_{1}\otimes n_{2})=\left(-1\right)^{p(n_{1})p(n_{2}^{*})}n_{1}^{*}(n_{1})n_{2}^{*}(n_{2})
\]
for $n_{i}\in N$ and $n_{i}^{*}\in N^{*}$.
\end{defn}
In the case where both $A$ and $M$ are purely even, we will denote
these elements by $\det(B)$ and $\det^{-1}(B)$ respectively. 
\begin{defn}
For $B$ a non-degenerate even symmetric form, we define an element
$\ori_{(0,1)}(B)\in\ori_{(0,1)}(M)$ by the following: a basis $\mathcal{\beta}=\{e_{i}\}\in\ori_{(0,1)}(B)$
if the matrix $B_{\beta}=\left(\begin{matrix}B_{11} & B_{12}\\
B_{21} & B_{22}
\end{matrix}\right)$ satisfies 
\[
\sgn(\pff(B_{22}))=1
\]
where $\left(B_{\beta}\right)_{ij}=B(e_{i},e_{j})$ and $ $$\pff$
denotes the Pfaffian of a skew-symmetric matrix.
\end{defn}
The main idea of this construction is that a symmetric bilinear form
is skew-symmetric on the odd part of the module in the ungraded sense,
and there defines an orientation the same way a symplectic form does
classically.

\subsection{Berezinian of an Odd Isomorphism\label{sub:Berezinian-of-an-odd-isomorphism}}

Let $A$ be a commutative super algebra and $M$ be a free $p|q$
module over $A$. Let $E$ be an odd automorphism of $M$. We will
construct an element of $\ber(M)^{\otimes2}$ from this automorphism.
We have that $E$ defines an isomorphism 
\[
\hat{E}:\Pi M\rightarrow M
\]
and therefore also 
\[
\ber(\hat{E}):\ber(\Pi M)\rightarrow\ber(M).
\]
Via the identification of $\ber(\Pi M)$ and $\ber(M^{*})$ the above
map becomes 
\[
\ber(\hat{E}):\ber(M^{*})\rightarrow\ber(M).
\]

\begin{defn}
We denote by $\ber(E)$ the image of 1 in $\ber(M)^{\otimes2}$ of
the map 
\[
A\cong\ber(M^{*})\otimes\ber(M)\xrightarrow{\ber(\hat{E})\otimes\id}\ber(M)^{\otimes2}.
\]

\end{defn}
If we were to perform an analogous construction for an even automorphism,
we would get an element of $\ber(M^{*})\otimes\ber(M)$ which is naturally
isomorphic to $A$.

In coordinates, $\ber(E)$ is given by the following expression: for
$\beta$ a basis of $M$, $b_{\mathcal{\beta}}\otimes b_{\mathcal{\beta}}$
is a basis element of $\ber(M)^{\otimes2}$ and we have 
\[
\ber(E)=\ber(E_{\beta}^{\beta}I')b_{\beta}\otimes b_{\beta}
\]
where 
\[
I'=\left(\begin{matrix}0 & \mbox{id}\\
\mbox{id} & 0
\end{matrix}\right)
\]
and $E_{\beta}^{\beta}$ is the matrix representation of $E$ in the
basis $\beta$.

\subsection{Square Root of One-Dimensional Modules\label{sub:Square-Root-of-1-d-modules}}

Let $A$ be a commutative super algebra over $\R$ with a sign homomorphism.
Suppose we are given an automorphism of the positive elements of $A$
\[
\sqrt{\mbox{ }}:A_{>0}^{\times}\rightarrow A_{>0}^{\times}
\]
such that $\left(\sqrt{a}\right)^{2}=a$ for all $a\in A_{>0}^{\times}.$
For example, we have such map for the algebra of functions on a supermanifold.%
\footnote{See the footnote on page \pageref{fn:Square_Root_on_Supermanifolds}.%
}
\begin{defn}
For any (even or odd) one-dimensional module $M$, we define the map
\[
\sqrt{\mbox{ }}:\left(M^{\otimes2}\right)_{\times}\rightarrow M\otimes\ori_{(1,1)}(M)
\]
where $M_{\times}$ denotes the set of elements of $M$ that constitute
a basis. If $v$ is a basis element of $M$, then $v\otimes v$ is
a basis element of $M^{\otimes2}$ and we define 
\[
\sqrt{fv\otimes v}:=\sqrt{\abs f}v\otimes\ori_{(1,1)}(v)
\]
where $f\in A_{0}^{\times}$ and $|f|=\sgn(f)\cdot f$ is positive.
\end{defn}

\section{Calculus on Supermanifolds\label{sec:Calculus-on-Supermanifolds}}

\subsection{Supermanifolds}

Just like many other geometric objects (manifolds, schemes, analytic
spaces), a supermanifold is a locally ringed space with a particular
local model. The local model in this case is the topological space
$\mathbb{R}^{m}$ with the structure sheaf $\mathcal{C}^{\infty}(\mathbb{R}^{m})[\theta^{1},\dots,\theta^{n}]$.
\begin{defn}
A supermanifold of dimension $m|n$ is a pair $M=(\overline{M},\mathcal{O}_{M})$
where $\overline{M}$ is a topological space and $\mathcal{O}_{M}$
is a sheaf of super algebras such that every point $m\in\overline{M}$
has a neighborhood $U\subset\overline{M}$ such that $(U,\mathcal{O}_{M}|_{U})\cong(\mathbb{R}^{m},\mathcal{C}^{\infty}(\mathbb{R}^{m})[\theta^{1},\dots,\theta^{n}])=:\mathbb{R}^{m|n}$.\end{defn}
\begin{example}
Let $\overline{M}$ be an ordinary $m$-manifold and $E\rightarrow\overline{M}$
a rank $n$ vector bundle. Then the pair $(M,\Gamma(\bigwedge^{*}E))$
is a supermanifold of dimension $m|n$ which we denote by $\Pi E^{*}$. 
\end{example}
A morphism of supermanifolds $f:M\rightarrow N$ is a continuous map
\[
|f|:\overline{M}\rightarrow\overline{N}
\]
and an even morphism of sheaves of rings over $\overline{M}$ 
\[
f^{*}:|f|^{-1}\mathcal{O}_{N}\rightarrow\mathcal{O}_{M}
\]
which on the level of stalks is a morphism of local rings. 

A supermanifold has an underlying ordinary manifold. Let $J\subset\mathcal{O}_{M}$
be the ideal generated by odd elements. The locally ringed space $(\overline{M},\mathcal{O}_{M}/J)$
is an ordinary smooth manifold since locally, the quotient $\mathcal{O}_{M}/J$
is isomorphic to $\mathcal{C}^{\infty}(\mathbb{R}^{m})$. This smooth
manifold is called the reduced manifold of $M$ and denoted $M_{\red}$.
We have an embedding 
\[
M_{\red}\hookrightarrow M
\]
in the category of supermanifolds. From now on we'll abandon the notation
$\overline{M}$ for the underlying topological space and denote it
by $M_{\red}$, i.e., $M=(M_{\red},\mathcal{O}_{M})$. 

A supermanifold can be given by charts and gluing data. The gluing
isomorphisms are maps of locally ringed spaces
\[
U\rightarrow V
\]
where $U,V\subset\mathbb{R}^{m|n}$. For ordinary manifolds, a map
of locally ringed spaces as above is specified by the pullbacks of
the coordinate functions of the image manifold. The same holds for
supermanifolds: a morphism 
\[
\phi:U\rightarrow V
\]
where $U\subset\mathbb{R}^{m|n},V\subset\mathbb{R}^{m'|n'}$ is uniquely
specified by $\phi^{*}(y^{i})$ and $\phi^{*}(\xi^{j})$ where $y^{i},\xi^{j}$
are the coordinate functions on $V$. Conversely, any collection of
$m'$ even functions $\psi_{i}$ and $n'$ odd functions $\eta_{j}$
on $U$ such that $(\psi_{1},\dots,\psi_{m'})\in V_{\red}\subset\mathbb{R}^{m'}$
when evaluated on $U_{\red}$ give rise to a morphism $\phi:U\rightarrow V$
such that $\phi^{*}(y^{i})=\psi_{i}$ and $\phi^{*}(\xi^{j})=\eta_{j}$.
A notation for this morphism resembling that of an ordinary manifold
is 
\[
\phi(x^{1},\dots,x^{m},\theta^{1},\dots,\theta^{n})=(\psi_{1},\dots,\psi_{m'},\eta_{1},\dots,\eta_{n'})
\]
where $x^{i},\theta^{j}$ are the coordinate functions on $U$. 

A vector bundle $E$ of rank $p|q$ over a supermanifold $M$ is a
sheaf of $\mathcal{O}_{M}$ modules over $M_{\red}$ which is locally
free of rank $p|q$. Given a morphism of supermanifolds $f:M\rightarrow N$,
the pullback vector bundle is defined by 
\[
f^{*}E=f^{-1}E\otimes_{f^{-1}\mathcal{O}_{N}}\mathcal{O}_{M}.
\]

\subsection{Vector Fields}

The tangent bundle $TM$ of a supermanifold is the sheaf of derivations
of the structure sheaf i.e. $\R$-linear maps $D:\cO\rightarrow\cO$
such that 
\[
D(ab)=D(a)b+(-1)^{p(a)p(D)}aD(b).
\]
Sections of the tangent bundle are referred to as vector fields. The
rank of $TM$ for a $(m|n)$-supermanifold $M$ is $(m|n)$ and in
a local chart with coordinates $x^{1},\dots,x^{m},\theta^{1},\dots,\theta^{n}$,
a basis of $TM$ is given by $\p{x^{i}},\p{\theta^{j}}$ where 
\begin{eqnarray*}
\p{x^{i}}(x^{j}) & = & \delta_{j}^{i};\ \ \ \ \ \ \ \ \ \p{x^{i}}(\theta^{j})=0,\\
\p{\theta^{i}}(\theta^{j}) & = & \delta_{j}^{i};\ \ \ \ \ \ \ \ \ \p{\theta^{i}}(x^{j})=0.
\end{eqnarray*}

A lot of geometric concepts having to do with vector fields carry
over from classical geometry once we are able to phrase them in terms
of maps of spaces and sheafs of sections. An important construction
is the flow of a vector field. In its most naive sense it is only
defined for even vector fields: If $V$ is an even vector field on
a supermanifold M, then there exists a flow map 
\[
\phi_{V}:D\rightarrow M
\]
where $D\subset M\times\mathbb{R}^{1|0}$ is an open neighborhood
of $M\times\{0\}$, having the property that for any function $f\in\cO_{M}$,
\[
\phi_{V}^{*}(V(f))=\p t\phi_{V}^{*}(f).
\]

The flow map allows us to define the Lie derivative. We will state
the definition for the Lie derivative of a vector field, but analogous
definitions work for other geometric structures that will be defined
later. Let $V,W\in TM$ be vector fields, and assume that $V$ is
even. Then the Lie derivative of $W$ with respect to $V$ is 
\[
\mathcal{L}_{V}W(f):=\frac{d}{dt}\mbox{\ensuremath{\bigg|}}_{t=0}W(i_{t}^{*}(\phi_{V}^{*}(f))
\]
where $i_{t}:M\rightarrow M\times\mathbb{R}^{1|0}$ is the embedding
$i_{t}(m)=(m,t)$.

To define the Lie derivative with respect to an odd vector field $Q$
on $M$, we consider the even vector field $\eta Q$ on $M\times\mathbb{R}^{0|1}$
where $\eta$ is the odd coordinate on $\R^{0|1}$ and define $\mathcal{L}_{Q}$
by requiring 
\[
\mathcal{L}_{\eta Q}(*)=\eta\mathcal{L}_{Q}(*).
\]

Similarly to the case of ordinary manifolds, the Lie derivative of
vector fields can be expressed as the Lie bracket where in the case
of supermanifolds we have to account for the Koszul sign rule. For
$V,W$ homogeneous vector fields, we define a new vector field $[V,W]$
of parity $p(V)+p(W)$ by 
\[
[V,W](f)=V(W(f))-(-1)^{p(V)p(W)}W(V(f))
\]
where $f$ is any function. We have 
\[
\mathcal{L}_{V}W=[V,W].
\]
Note that for an odd vector field $Q$, we have $[Q,Q]=2Q^{2}$ which
in general need not vanish.

A $(i,j)$-orientation of a supermanifold is a section of the bundle
$\ori_{(i,j)}(TM)$. A supermanifold might be $(i,j)$-orientable
for some pairs $(i,j)$ but not others. For example, a $(1,0)$-orientation
of $M$ is the same as orientation of $M_{red}$ while a $(0,1)$-orientation
of $\Pi E$ where $E\rightarrow N$ is an ordinary vector bundle,
is an orientation of $E$.

\subsection{Integration}

One can define differential forms on supermanifolds as sections of
exterior powers of the cotangent bundle just as for ordinary manifolds,
but if a supermanifold has non-zero odd dimension, then there are
differential forms of arbitrary high degree. In particular, there
is not a top degree form that one can integrate over a supermanifold.
In order to get an integration theory, we need to generalize the top
degree form in a different manner. For an ordinary vector space $E$,
the top degree exterior power is isomorphic to the determinant line.
The generalization of the determinant to supergeometry is the Berezinian. 
\begin{defn}
A section of $\ber(TM^{*})$ is called an integral form. A section
of $\ber(TM^{*})\otimes\ori_{(1,0)}(TM)$ is called a density.
\end{defn}
Suppose first that $\mu$ is a compactly supported integral form on
$\R^{m|n}$. Let $x^{1},\dots,x^{m},\theta^{1},\dots,\theta^{n}$
be the coordinates on $\R^{m|n}$. A basis $\beta$ of $TM^{*}$ is
given by $\beta=\{dx^{1},\dots,d\theta^{n}\}$ which is dual to the
basis $\{\p{x^{1}},\dots,\p{\theta^{n}}\}$ of $TM$. The basis of
$\ber(TM^{*})$ that we denoted $b_{\beta}$ in  \subref{The-Berezinian}
is denoted by $[dx^{1},\dots,d\theta^{n}]$ in this context. We can
thus write the integral form $\mu$ as 
\[
\mu=[dx^{1},\dots,d\theta^{n}]\left(\sum_{I\subset\{1,\dots,n\}}f_{I}\theta^{I}\right)
\]
where $\theta^{I}=\prod_{i\in I}\theta^{i}$ and $f_{I}\in\mathcal{C}_{c}^{\infty}(\mathbb{R}^{m})$.
We define the integral of $\mu$ over $\R^{m|n}$ to be
\[
\int_{\R^{m|n}}\mu=\int_{\R^{m|n}}[dx^{1},\dots,d\theta^{n}]\left(\sum_{I\subset\{1,\dots,n\}}f_{I}\theta^{I}\right)=\int_{\R^{m}}dx^{1}\dots dx^{m}f_{\{1,2,\dots,n\}}
\]
where the right side is an ordinary integral. We refer to \cite{GaugeFieldTheory}
for a careful proof that this integral does not depend on the chosen
coordinates $x^{1},\dots,\theta^{n}$, but only on the induced orientation
on $\mathbb{R}^{m}$ or in other words on $\ori_{(1,0)}(\R^{m|n})$.

Now suppose $M$ is a $(1,0)$-oriented supermanifold and $\mu$ a
compactly supported form on $M$. To calculate the integral of $\mu$,
we find a locally finite cover of $M$ by oriented affine charts $\{U_{i}\}$
and a partition of unity $\{\phi_{i}\}$ subordinate to $\{U_{i}\}$.
The integral is defined by 
\[
\int_{M}\mu=\sum_{i}\int_{U_{i}}\phi_{i}\mu
\]
where the sum on the right is finite since $\mu$ is compactly supported
and $\{U_{i}\}$ is locally finite. The arguments for why the partition
of unity exists and why this definition does not depend on the choices
made is analogous to the classical case. 

Instead of assuming that $M$ is $(1,0)$-oriented, we can instead
integrate a density. In particular, a density can be integrated over
a supermanifold regardless of whether it is $(1,0)$-orientable. 

One defines the Lie derivative of an integral form analogously to
how we defined the Lie derivative of a vector field. Of importance
in future sections is the following fact: if $\mu$ is a compactly
supported density on a supermanifold and $V$ is a vector field, then
\[
\int_{M}\mathcal{L}_{V}\mu=0.
\]
This statement follows from the fact that integration is invariant
under diffeomorphisms.

\subsection{Subsupermanifolds}

Let $M$ be a supermanifold. A closed subsupermanifold $N$ is given
by a morphism $i:N\rightarrow M$ such that restricted to the reduced
manifolds, $i$ is a closed embedding and $di:TN\rightarrow i^{*}TM$
is an inclusion of a direct summand. Equivalently, a closed subsupermanifold
can be given by a sheaf of ideals $I$ of $\mathcal{O}_{M}$ such
that for every point $p\in M_{\red}$, there exists a coordinate chart
$U\ni p$ with coordinates $\{x^{1},\dots,x^{m},\theta^{1},\dots,\theta^{n}\}$
such that $I|_{U}=\langle x^{1},\dots,x^{k},\theta^{1},\dots,\theta^{l}\rangle$
for some $k,l\in\N$. One can then define $N_{\red}\subset M_{\red}$
as the set of points where the stalks of $I$ are proper ideals of
$\mathcal{O}_{M}$ and the structure sheaf $\mathcal{O}_{N}$ as $i_{\red}^{-1}\left(\mathcal{O}_{M}/I\right)$.
Conversely, given an embedding $i:N\rightarrow M$ as above, the sheaf
of ideals defined by 
\[
I(U)=\{\psi\in\mathcal{O}_{M}|i^{*}\phi=0\}
\]
satisfies the desired property.

\subsection{Vanishing Subsupermanifolds of Sections of Vector Bundles}

Let $E$ be a vector bundle on $M$ and $s$ a section of $E$. We
define the vanishing space of $s$. We work locally on $M_{\red}$.
Given a point $p\in M_{\red}$ there exists an open set $U\ni p$
such that $E$ is trivial when restricted to $U$. Pick $e_{1},\dots,e_{k}$,
sections of $E|_{U}$ trivializing the bundle. A section $s$ can
then be written as
\[
s=\sum_{i=1}^{k}e_{i}s_{i}
\]
where $s_{i}\in\mathcal{O}_{U}$. We would like to define the vanishing
space of $s$ by the ideal 
\[
I_{s}:=\langle s_{1},\dots,s_{k}\rangle\subset\mathcal{O}_{U}.
\]
The ideal $I_{s}$ does not depend on the choice of trivialization
of $E$. In general though, $I_{s}$ does not define a subsupermanifold
i.e. it is not locally generated by coordinate functions. 
\begin{defn}
For $s$ a section of a vector bundle, we say that the vanishing space
of $s$ is a subsupermanifold if the ideal $I_{s}$ defines a subsupermanifold
$N\subset M$. In that case we call $N$ the vanishing subsupermanifold
of $s$.
\end{defn}
There is a more direct way to define the ideal $I_{Q}$ for a vector
field $Q$ on $M$. In particular, $I_{Q}$ is generated by $Q(\cO_{M})$.
Assume that the vanishing space of $Q$ is a subsupermanifold $N$.
Then, just as classically, $\mathcal{L}_{Q}$ gives rise to a linear
map
\[
L:TM|_{N}\rightarrow TM|_{N}
\]
where $TM|_{N}$ is the pullback $i^{*}TM$ where $i:N\hookrightarrow M$
is the inclusion of the subsupermanifold $N$. Moreover, $L$ vanishes
on $TN\subset TM$ and therefore descends to a map 
\[
TM/TN=:\nu_{N}\rightarrow TM|_{N}\rightarrow\nu_{N}.
\]

\begin{defn}
Let $Q$ be a vector field on $M$ and assume that the vanishing space
of $Q$ is a subsupermanifold $N.$ We call the vanishing space non-degenerate
if the map $L:\nu_{N}\ra\nu_{N}$ induced by the Lie derivative $\cL_{Q}$
is an automorphism. 
\end{defn}
The critical space of a function $\psi\in\cO_{M}$ is defined as the
vanishing space of the 1-form $d\phi$. Similarly to the case of vector
fields, the ideal $I_{d\phi}$ is generated by $\{Q\psi|Q\in TM\}$.
Assume that the critical space of $\psi$ is a subsupermanifold $N$.
The Hessian of $\psi$ is a symmetric bilinear form on $TM|_{N}$
defined by 
\[
\hess(\psi)(X,Y)=X'Y'\psi|_{N}
\]
where $X',Y'$ are any vector fields on $M$ which restrict to $X,Y$
on $N$. The Hessian vanishes on $TN$ and therefore descends to the
normal bundle $\nu_{N}$.
\begin{defn}
Let $\psi$ be a function on a supermanifold $M$ whose critical subspace
is a subsupermanifold $N$. We say that $N\subset M$ is non-degenerate
if $\hess(\psi)$ restricted to $\nu_{N}$ is a non-degenerate bilinear
form.
\end{defn}

\section{Stationary Phase Approximation\label{sec:Stationary-Phase-Approximation}}

\subsection{Stationary Phase for Ordinary Manifolds}

We first state the result in 1-dimension.
\begin{prop}
\label{prop:1-d_stationary_phase} Let $f$ be a compactly supported
function on $\mathbb{R}$ and $S$ a smooth function on $\mathbb{R}$
having a single non-degenerate critical point at $0$. We then have
\[
\int fe^{i\lambda S}dx=e^{i\lambda S(0)}\frac{1}{\sqrt{\abs{S''(0)}}}e^{sgn(S''(0))\frac{\pi}{4}i}\left(\frac{2\pi}{\lambda}\right)^{\frac{1}{2}}\left(f(0)+O(\lambda^{-1})\right)
\]
as $\lambda\ra\infty$.
\end{prop}
For a proof, see for example \cite{asymptotics,bott_critical_point}.
In \cite{bott_critical_point}, the higher order terms are also computed.
To extend this result to higher dimensional manifolds, we need the
Morse-Bott Lemma. See \cite{morse_bott} for a proof.
\begin{prop}[Morse-Bott Lemma]
\label{prop:Morse-Bott} Let $S$ be a function on a manifold of
dimension $n$ such that its critical space is a non-degenerate submanifold
$N$. Then around each point $p\in N$, there exists a coordinate
system $\{x^{i}\}_{i=1}^{n}$ on a neighborhood $U\ni p$ in which
$S$ is expressed as 
\[
S=S(p)+\left(x^{1}\right)^{2}+\dots+\left(x^{\alpha}\right)^{2}-\left(x^{\alpha+1}\right)^{2}-\dots-\left(x^{k}\right)^{2}
\]
where $k$ is the codimension of $N$. 
\end{prop}
Under the assumptions of the proposition, the Hessian $H$ of $S$
defines a non-degenerate bilinear form on the normal bundle $\nu_{N}$
of $N$ and therefore an element
\[
\sqrt{\mbox{det}^{-1}(H)}\in\det(\nu_{N})\otimes\ori(\nu_{N})
\]
as explained in \subref{Bilinear-Forms} and \subref{Square-Root-of-1-d-modules}.

Since we have the exact sequence 
\[
0\rightarrow TN\rightarrow TM\rightarrow\nu_{N}\rightarrow0,
\]
we have that 
\[
\det TM^{*}\otimes\ori(TM)\cong\left(\det TN^{*}\otimes\ori(TN)\right)\otimes\left(\det\nu_{N}^{*}\otimes\ori(\nu_{N})\right).
\]
In particular, contracting with the section $\sqrt{\mbox{det}^{-1}(H)}$,
gives a map
\[
\inner{\sqrt{\mbox{det}^{-1}(H)},\bullet}:\det TM^{*}\otimes\ori(TM)\rightarrow\det TN^{*}\otimes\ori(TN)
\]
from densities on $M$ to densities on $N$.
\begin{prop}
\label{prop:stationary_phase} Let $M$ be a compact manifold and
$\mu$ a density. Let $S$ be a smooth function such that its critical
space is a non-degenerate connected submanifold $N$ of codimension
$k$ and denote $H=\hess(S)$. We have 
\[
\int_{M}\mu e^{i\lambda S}=e^{\sgn(H)\frac{\pi}{4}i}\left(\frac{2\pi}{\lambda}\right)^{\frac{k}{2}}\left(\int_{N}e^{i\lambda S}\inner{\sqrt{\mbox{det}^{-1}H},\mu}+O(\lambda^{-1})\right)
\]
as $\lambda\ra\infty$ where $\sgn(H)$ is the signature of $H$ (the
number of positive eigenvalues minus the number of negative eigenvalues). \end{prop}
\begin{proof}
If $S$ has no critical points in the support of $\mu$ then $\int_{M}\mu e^{i\lambda S}=O(\lambda^{-\infty})$.
To see it, note that there exists a vector field $V$ such that $V(S)$
is invertible on the support of $\mu$. By invariance of integration
under diffeomorphisms we have 
\[
0=\int\cL_{V}(\mu\frac{1}{i\lambda V(S)}e^{i\lambda S})=\frac{1}{i\lambda}\int\cL_{V}\left(\frac{\mu}{V(S)}\right)e^{i\lambda S}+\int\mu e^{i\lambda S}.
\]
We have $ $$\abs{\int\cL_{V}\left(\frac{\mu}{VS}\right)e^{i\lambda S}}\leq C$
for some constant $C$ since $\abs{e^{i\lambda S}}=1$ for all $\lambda$
and therefore $\int\mu e^{i\lambda S}=O(\lambda^{-1})$. Applying
the same argument $l$ times we get $\int_{M}\mu e^{i\lambda S}=O(\lambda^{-l})$
for any $l\in\N$.

This shows that the asymptotic behavior of $Z(\lambda)$ only depends
on a neighborhood of $N$. The actual expression is gotten by a simple
application of partition of unity, the Morse-Bott Lemma (\propref{Morse-Bott})
and the statement in the case of one dimension (\propref{1-d_stationary_phase}).
\end{proof}
As an example, assume that $N$ is an isolated point and we are working
in a fixed coordinate system $\{x^{1},\dots,x^{n}\}$. Then $\inner{\sqrt{\mbox{det}^{-1}H},dx^{1}\wedge\dots\wedge dx^{n}}=\frac{1}{\sqrt{\abs{\det H}}}$
and the first degree approximation is easily computed.

\subsection{Stationary Phase for Supermanifolds}

Let $M$ be a supermanifold, $\mu$ a compactly supported density
and $S$ an even function. We again study the asymptotic behavior
of 
\[
Z(\lambda)=\int_{M}\mu e^{i\lambda S}.
\]

If $dS$ is nowhere vanishing, then $Z(\lambda)=O(\lambda^{-\infty})$
for the same reason as in the proof of \propref{stationary_phase}.
We proceed in the same way as we did in the classical case. We first
show that near a non-degenerate critical subsupermanifold, an even
function $S$ has the standard form in appropriate coordinates. We
will then write a formula for stationary phase approximation that
is easy to verify in standard coordinates.
\begin{thm}
\label{thm:super_morse_bott1} Let $S$ be an even function on a supermanifold
of dimension $n|m$ such that its critical space is a non-degenerate
subsupermanifold $N$ of codimension $k|l'$. Then $l'=2l$ is even
and for each point $p\in N_{\red}$, there exists a coordinate system
$\{x^{1},\dots,x^{n},\theta^{1},\dots,\theta^{m}\}$ on a neighborhood
$U\ni p$ in which $S$ is expressed as
\[
S=S(p)+\left(x^{1}\right)^{2}+\dots+\left(x^{\alpha}\right)^{2}-\left(x^{\alpha+1}\right)^{2}-\dots-\left(x^{k}\right)^{2}+\theta^{1}\theta^{2}+\dots+\theta^{2l-1}\theta^{2l}.
\]
\end{thm}
\begin{proof}
To see that $l'$ is even, consider the vector bundle $j^{*}(\nu_{N})$
where $j:N_{red}\hookrightarrow N$ and $\nu_{N}$ is the normal bundle
to $N$. This bundle has a canonical decomposition into even and odd
submodules since $\mathcal{O}_{N_{red}}$ is purely even. In particular,
the restriction of $H:=\hess(S)$ to $j^{*}(\nu_{N})_{1}$ is a non-degenerate
skew-symmetric bilinear form since $H$ is non-degenerate. This forces
$j^{*}(\nu_{N})_{1}$ to be even dimensional which implies that $l'$
is even.

We prove the theorem by induction where we induct on $a\in\mathbb{N}$
in the statement that there exists a coordinate system $\{x^{1},\dots,x^{n},\theta^{1},\dots,\theta^{m}\}$
on a neighborhood $U\ni p$ such that 
\[
S=S(p)+\left(x^{1}\right)^{2}+\dots+\left(x^{\alpha}\right)^{2}-\left(x^{\alpha+1}\right)^{2}-\dots-\left(x^{k}\right)^{2}+\theta^{1}\theta^{2}+\dots+\theta^{2l-1}\theta^{2l}\mod J^{2a}
\]
where $J$ is the ideal generated by odd functions. The theorem is
proved by noting that when $a>\frac{m}{2}$, the ideal $J^{2a}$ is
the zero ideal. We first consider the base case where $a=2$.

Pick a coordinate system $\{x^{1},\dots,x^{n},\theta^{1},\dots,\theta^{m}\}$
on a neighborhood $U$ of $p\in N$ in which $N$ is given by vanishing
of the ideal $I=\langle x^{1},\dots,x^{k},\theta^{1},\dots,\theta^{2l}\rangle$.
By the Morse-Bott Lemma (\propref{Morse-Bott}), by changing the even
coordinates and subtracting the constant $S(p)$, we may assume that
\[
S=\left(x^{1}\right)^{2}+\dots+\left(x^{\alpha}\right)^{2}-\left(x^{\alpha+1}\right)^{2}-\dots-\left(x^{k}\right)^{2}+\sum_{i,j=1}^{m}S_{ij}\theta^{i}\theta^{j}\mod J^{4}.
\]
where $(S_{ij})$ is a skew-symmetric matrix of functions on $U_{\red}.$
Non-degeneracy of $N$ implies that $(S_{ij})_{i,j=1}^{2l}$ is non-degenerate
on $N_{\red}\cap U_{\red}$. Since non-degeneracy is an open condition
we may take $U$ small enough so that $(S_{ij})_{i,j=1}^{2l}$ is
non-degenerate on $U_{\red}$. By a method analogous to the Gram-Schmidt
process, every symplectic vector space possesses a symplectic basis.
By a parametrized version of the same method, a symplectic vector
bundle locally possesses a symplectic basis of sections. Applying
this to the trivial $\R^{2l}$ bundle over $U_{\red}$ with the symplectic
form $(S)_{i,j=1}^{2l}$, there exists a non-degenerate matrix $(T_{ij})_{i,j=1}^{2l}$
of functions on $U_{\red}$ such that 
\[
T^{T}ST=\frac{1}{2}\left(\begin{matrix}0 & -1\\
1 & 0\\
 &  & \ddots\\
 &  &  & 0 & -1\\
 &  &  & 1 & 0
\end{matrix}\right).
\]
By changing the odd coordinates $\theta^{1},\dots,\theta^{2l}$ in
$U$ by 
\[
\theta_{{\rm old}}^{i}=\sum_{i=1}^{2l}T_{ij}\theta_{{\rm new}}^{i}
\]
we have 
\[
S=\left(x^{1}\right)^{2}+\dots+\left(x^{\alpha}\right)^{2}-\left(x^{\alpha+1}\right)^{2}-\dots-\left(x^{k}\right)^{2}+\theta^{1}\theta^{2}+\dots+\theta^{2l-1}\theta^{2l}+\sum_{\substack{i=1\\
j=2l+1
}
}^{m}S_{ij}'\theta^{i}\theta^{j}\mod J^{4}
\]
for some functions $S_{ij}^{'}$ such that $S_{ij}'=-S_{ji}^{'}$
if $i,j>2l$.

To simplify notation, define a function $S_{st}$ of $k$ even variables
and $2l$ odd variables by 
\[
S_{st}(x^{1},\dots,\theta^{2l})=\left(x^{1}\right)^{2}+\dots+\left(x^{\alpha}\right)^{2}-\left(x^{\alpha+1}\right)^{2}-\dots-\left(x^{k}\right)^{2}+\theta^{1}\theta^{2}+\dots+\theta^{2l-1}\theta^{2l}
\]
so that we have 
\begin{equation}
S=S_{st}(x^{1},\dots,\theta^{2l})+\sum_{\substack{i=1\\
j=2l+1
}
}^{m}S_{ij}'\theta^{i}\theta^{j}\mod J^{4}.\label{eq:S=00003DS_st+...}
\end{equation}

Assume without loss of generality that $N$ is connected. The function
$S$ vanishes on $N_{\red}$ and since it is constant on $N$, it
vanishes on $N$. Since both $S$ and $dS$ vanish on $N$, it follows
that $S\in I^{2}$ where $I=\langle x^{1},\dots,x^{k},\theta^{1},\dots,\theta^{2l}\rangle$
is the ideal of functions vanishing on $N$. The choice of coordinates
$\{x^{1},\dots,x^{n},\theta^{1},\dots,\theta^{m}\}$ introduces a
grading by $\left(\Z/2\Z\right)^{m}$ on $\cO_{M}$ in which functions
of degree $\{d_{i}\}_{i=1}^{m}\in\left(\Z/2\Z\right)^{m}$ are those
of the form $f(x^{1},\dots,x^{n})\prod_{i=1}^{m}\left(\theta^{i}\right)^{d_{i}}.$
The ideal $I^{2}$ is a graded ideal with respect to this grading
and therefore by decomposing $S$ into its homogeneous components,
each term $S_{ij}'\theta^{i}\theta^{j}$ in \eqref{S=00003DS_st+...}
belongs to $I^{2}$. We thus have $S_{ij}'$ vanishes on $N_{\red}\cap U_{\red}$
at least to second order if $i>2l$ and at least to first order if
$i\leq2l$. Thus

\[
S=S_{st}(x^{1},\dots,\theta^{2l})+\sum_{\beta,\gamma=1}^{k}x^{\beta}x^{\gamma}F_{\beta\gamma}+\sum_{i=1}^{2l}G_{i}\theta^{i}\mod J^{4}
\]
for some functions $F_{\beta\gamma}\in J^{2}$ and $G_{i}\in J\cap\langle x^{1},\dots,x^{k}\rangle$
where 
\[
\sum_{\beta,\gamma=1}^{k}x^{\beta}x^{\gamma}F_{\beta\gamma}=\sum_{\substack{i=2l+1\\
j=2l+1
}
}^{m}S_{ij}'\theta^{i}\theta^{j},
\]
\[
\sum_{i=1}^{2l}G_{i}\theta^{i}=\sum_{i=1}^{2l}\sum_{j=2l+1}^{m}S_{ij}'\theta^{i}\theta^{j}.
\]

Consider the following change of coordinates:
\[
\tilde{\theta}^{i}=\begin{cases}
\theta^{i}+(-1)^{i}G_{i+(-1)^{i+1}} & i\leq2l\\
\theta^{i} & i>2l
\end{cases}.
\]
Since $G_{i}\in\langle x^{1},\dots,x^{k}\rangle$, up to taking a
smaller neighborhood $U\ni p$, the above is an invertible transformation
on $U$ and 
\[
S=S_{st}(x^{1},\dots,\tilde{\theta}^{2l})+\sum_{\beta,\gamma=1}^{k}x^{\beta}x^{\gamma}F'_{\beta\gamma}\mod J^{4}
\]
for some functions $F'_{\beta\gamma}\in J^{2}$ satisfying $F'_{\beta\gamma}=F'_{\gamma\beta}$
. Consider the following change of coordinates: 
\[
\tilde{x}^{\beta}=\begin{cases}
x^{\beta}+\frac{1}{2}\sum_{\gamma=1}^{k}x^{\gamma}F'_{\beta\gamma} & \beta\leq\alpha\\
x^{\beta}-\frac{1}{2}\sum_{\gamma=1}^{k}x^{\gamma}F'_{\beta\gamma} & \alpha<\beta\leq k\\
x^{\beta} & k<\beta
\end{cases}.
\]
This too is an invertible transformation on a neighborhood of $p$
and 
\[
S=S_{st}(\tilde{x}^{1},\dots,\tilde{\theta}^{2l})\mod J^{4}.
\]
This establishes the base case for the induction.

We now consider the inductive step. Assume there exists a coordinate
system $\{x^{1},\dots,x^{n},\theta^{1},\dots,\theta^{m}\}$ on a neighborhood
$U\ni p$ such that 
\[
S=S_{st}(x^{1},\dots,\theta^{2l})\mod J^{2a}
\]
with $a\geq2$. As above, since $S\in I^{2}$, we have
\[
S=S_{st}(x^{1},\dots,\theta^{2l})+\sum_{\beta,\gamma=1}^{k}x^{\beta}x^{\gamma}F_{\beta\gamma}+\sum_{i=1}^{2l}G_{i}\theta^{i}\mod J^{2(a+1)}
\]
for some functions $F_{\beta\gamma}\in J^{2a}$ and $G_{i}\in J^{2a-1}$
. Changing the coordinates by
\begin{eqnarray*}
\tilde{x}^{\beta} & = & \begin{cases}
x^{\beta}+\frac{1}{2}\sum_{\gamma=1}^{k}x^{\gamma}F_{\beta\gamma} & \beta\leq\alpha\\
x^{\beta}-\frac{1}{2}\sum_{\gamma=1}^{k}x^{\gamma}F_{\beta\gamma} & k\leq\beta>\alpha\\
x^{\beta} & \beta>k
\end{cases}
\end{eqnarray*}
\[
\tilde{\theta}^{i}=\begin{cases}
\theta^{i}+(-1)^{i}G_{i+(-1)^{i+1}} & i\leq2l\\
\theta^{i} & i>2l
\end{cases}
\]
we have
\[
S=S_{st}(\tilde{x}^{1},\dots,\tilde{\theta}^{2l})\mod J^{2(a+1)}.
\]
We have used that $G_{i}G_{j}\in J^{4a-2}\subset J^{2(a+1)}$ which
is only true when $a\geq2$. This is why we had to be more careful
when discussing the base case. \end{proof}
\begin{thm}
\label{thm:super_stationary_phase_approx} Let $S$ be an even function
on a supermanifold $M$ such that its critical space is a non-degenerate
connected subsupermanifold $N$ of codimension $k|2l$. Let $\mu\in\ber(TM^{*})\otimes\ori_{(1,0)}$
be a compactly supported density and define $Z(\lambda):=\int_{M}\mu e^{i\lambda S}$.
Denote by $H$ the Hessian of $S$ and by $H_{red}$ the Hessian of
$S|_{M_{\red}}$. We have 
\[
Z(\lambda)=e^{i\lambda S(N_{red})}e^{\frac{\pi}{4}i\cdot\sgn(H_{\red})}\left(\frac{2\pi}{\lambda}\right)^{\frac{k}{2}}\left(-i\lambda\right)^{l}\int_{N}\inner{\sqrt{\ber^{-1}H}\otimes\ori_{(0,1)}(H),\mu}\cdot\left(1+O(\lambda^{-1})\right)
\]
where $\sgn(H_{\red})$ is the signature of $(H_{\red})$. 
\end{thm}
To clarify statement of the theorem above, recall from \subref{Bilinear-Forms}
and \subref{Square-Root-of-1-d-modules} that $\sqrt{\ber^{-1}H}\in\ber(\nu_{N})\otimes\ori_{(1,1)}(\nu_{N})$
and therefore %
\footnote{\label{fn:Square_Root_on_Supermanifolds}Recall that to define $\sqrt{\ber^{-1}H}$
we need a homomorphism $\sgn:\cO_{M}^{\times}\ra\{\pm1\}$ and a square
root operator $\sqrt{\ }:\left(\cO_{M}^{\times}\right)_{>0}\ra\left(\cO_{M}^{\times}\right)_{>0}$.
The sign is given by restricting to $M_{\red}$. The square root is
given by the following local formula: picking a system of coordinates,
a positive function is given by $f=f_{0}+\psi$ where $f_{0}$ is
a positive function on $M_{\red}$ and $\psi\in J$ where $J$ is
the ideal generated by odd functions. We then have 
\[
\sqrt{f}=\sqrt{f_{0}}\left(\sqrt{1+\frac{\psi}{\sqrt{f_{0}}}}\right)=\sqrt{f_{0}}\sum_{n=0}^{\infty}\frac{(-1)^{n}(2n)!}{(1-2n)(n!)^{2}(4^{n})}\left(\frac{\psi}{\sqrt{f_{0}}}\right)^{n}.
\]
The above sum is finite since $\psi$ is nilpotent. One can check
that $\sqrt{f}$ is the unique positive function satisfying $\left(\sqrt{f}\right)^{2}=f$
and hence does not depend on the chosen coordinates.%
} 
\[
\sqrt{\ber^{-1}H}\otimes\ori_{(0,1)}(H)\in\ber(\nu_{N})\otimes\ori_{(1,0)}(\nu_{N}).
\]
The exact sequence 
\[
0\rightarrow TN\rightarrow TM\rightarrow\nu_{N}\rightarrow0
\]
gives an isomorphism 
\[
\ber(TM^{*})\otimes\ori_{(1,0)}(TM)\cong\left(\ber(\nu_{N}^{*})\otimes\ori_{(1,0)}(\nu_{N})\right)\otimes\left(\ber(TN^{*})\otimes\ori_{(1,0)}(TN)\right).
\]
The contraction $\inner{\sqrt{\ber^{-1}H}\otimes\ori_{(0,1)}(H)),\mu}\in\ber(TN^{*})\otimes\ori_{(1,0)}(TN)$
is then a density on $N$. 

We note that a similar statement with the integrand being real, i.e.
$\mu e^{-\lambda S}$, is discussed in \cite{Schwarz_Semiclassical_approximation}.
\begin{proof}
It suffices to check this proposition locally on $N$ with the general
statement following by using a partition of unity. By the Morse-Bott
Lemma generalized to supermanifolds (\thmref{super_morse_bott1}),
there exists a coordinate system in which 
\[
S=C+\left(x^{1}\right)^{2}+\dots+\left(x^{\alpha}\right)^{2}-\left(x^{\alpha+1}\right)^{2}-\dots-\left(x^{k}\right)^{2}+\theta^{1}\theta^{2}+\dots+\theta^{2l-1}\theta^{2l}
\]
where $C=S(N_{\red})$. To simplify notation, let 
\begin{eqnarray*}
S_{0} & = & C+\left(x^{1}\right)^{2}+\dots+\left(x^{\alpha}\right)^{2}-\left(x^{\alpha+1}\right)^{2}-\dots-\left(x^{k}\right)^{2}\\
S_{1} & = & \theta^{1}\theta^{2}+\dots+\theta^{2l-1}\theta^{2l}\\
S & = & S_{0}+S_{1}.
\end{eqnarray*}
Assume that $\mu$ is supported in this coordinate chart and thus
can be expressed as $\mu=[dxd\theta]f$. We have
\[
Z(\lambda)=\int_{\mathbb{R}^{p|q}}[dxd\theta]f\cdot(1+i\lambda S_{1}+\dots+\frac{(i\lambda S_{1})^{l}}{l!})e^{i\lambda S_{0}}.
\]
The highest degree term in $\lambda$ is $f\frac{(i\lambda S_{1})^{l}}{l!}e^{i\lambda S_{0}}$.
Since $\frac{S_{1}^{l}}{l!}=\theta^{1}\dots\theta^{2l}$, the integration
of this term picks up the coefficient of $\theta^{2l+1}\theta^{2l+2}\dots\theta^{q}$
in $f$. We have 
\[
Z(\lambda)=(i\lambda)^{l}\int_{\mathbb{R}^{p}}dx^{1}\dots dx^{p}\left(\p{\theta^{q}}\dots\p{\theta^{2l+1}}f\right)_{|\mathbb{R}^{p}}e^{i\lambda S_{0}}\left(1+O(\lambda^{-1})\right)
\]
and therefore applying the ordinary stationary phase approximation
(\propref{Morse-Bott}) we get 
\begin{eqnarray*}
Z(\lambda) & = & e^{i\lambda C}e^{\frac{\pi}{4}i\left(\alpha-(k-\alpha)\right)}\left(\frac{2\pi}{\lambda}\right)^{\frac{k}{2}}(i\lambda)^{l}\int_{\mathbb{R}^{p-k}}dx^{k+1}\dots dx^{p}\left(\p{\theta^{q}}\dots\p{\theta^{2l+1}}f\right)_{|\mathbb{R}^{p-k}}\cdot(1+O(\lambda^{-1}))\\
 & = & e^{i\lambda C}e^{\frac{\pi}{4}i\cdot\sgn(H_{\red})}\left(\frac{2\pi}{\lambda}\right)^{\frac{k}{2}}(-i\lambda)^{l}\int_{N}\inner{\sqrt{\ber^{-1}H}\otimes\ori_{(0,1)}(H),[dxd\theta]f}\cdot(1+O(\lambda^{-1})).
\end{eqnarray*}
The last equality is simply a calculation of $\inner{\sqrt{\ber^{-1}H}\otimes\ori_{(0,1)}(H),[dxd\theta]f}$
in local coordinates. The factor of $(-1)^{l}$ in the above equation
comes from the fact that the matrix of $H$ in the basis $\p{x^{1}},\dots,\p{x^{k}},\p{\theta^{1}},\dots\p{\theta^{2l}}$
of $\nu_{N}$ is 
\[
\left(\begin{matrix}1\\
 & \ddots\\
 &  & 1\\
 &  &  & 0 & -1\\
 &  &  & 1 & 0\\
 &  &  &  &  & \ddots\\
 &  &  &  &  &  & 0 & -1\\
 &  &  &  &  &  & 1 & 0
\end{matrix}\right).
\]
Since 
\[
\pff\left(\begin{matrix}0 & -1\\
1 & 0
\end{matrix}\right)=-1,
\]
the orientation $\ori_{(0,1)}(H)$ differs from the one given by $\p{\theta^{1}},\dots,\p{\theta^{2l}}$
by $(-1)^{l}$. 
\end{proof}

\section{Localization Theorem\label{sec:Localization-Theorem}}

\subsection{General Localization Statement\label{sub:General-Localization-Statement}}

The constructions and proofs in this subsection are taken from \cite{SuppersymmetryAndLocalization}.
We include the proofs for completeness.

Let $M$ be a supermanifold. We begin with a series of definitions.
\begin{defn}
We call an even vector field $V$ on $M$ compact if there exists
an action by a compact torus $T$ on $M$ and an element $X\in\mathfrak{\t}$
in the Lie algebra of $T$ which induces the vector field $V$.
\end{defn}

\begin{defn}
We denote the normal bundle to $M_{red}$ by $\alpha_{M}$. It is
a vector bundle on the ordinary manifold $M_{\red}$.
\end{defn}
An odd vector field $Q$ on $M$ naturally defines a section $\alpha(Q)$
of $\alpha_{M}$. Similarly, an odd function $\sigma$ defines a section
$\alpha^{*}(\sigma)$ of the conormal bundle $\alpha_{M}^{*}$ induced
by the 1-form $d\sigma$. If $\{x^{1},\dots,x^{n},\theta^{1},\dots,\theta^{m}\}$
are coordinates in some neighborhood such that an odd vector field
$Q$ is given by 
\[
Q=\sum_{\alpha=1}^{m}(b^{\alpha}(x)+\dots)\p{\theta^{\alpha}}+\sum_{i=1}^{n}(\dots)\p{x^{i}}
\]
where the terms in $(\dots)$ belong to the ideal $J$ generated by
odd functions, then $\alpha(Q)=\sum_{\alpha=1}^{m}b^{\alpha}\p{\theta^{\alpha}}$.
The section $\alpha^{*}(\sigma)$ for $\sigma$ an odd function is
defined similarly from an expression of $d\sigma$.
\begin{defn}
For $Q$ an odd vector field, let $R_{Q}\subset M_{\red}$ be the
vanishing set of $\alpha(Q)$.
\end{defn}
Note that $R_{Q}$ need not have a nice form for an arbitrary $Q$.
\begin{prop}[{\cite[Theorem 1]{SuppersymmetryAndLocalization}}]
\label{prop:general_localization_theorem}Let $M$ be a supermanifold,
$Q$ an odd vector field and $\mu\in\ber(TM^{*})\otimes\ori_{(1,0)}(TM)$
a compactly supported density. If $\mu$ is $Q$-invariant ($\mathcal{L}_{Q}\mu=0$),
and $Q^{2}$ is compact then for any neighborhood $U\supset R_{Q}$
in $M_{red}$, there exists an even $Q$-invariant function $g$ which
takes the value $1$ on some neighborhood $O\supset R_{Q}$ and vanishes
outside $U$. Moreover, for every such $g$ we have 
\[
\int_{M}\mu=\int_{M}\mu\cdot g.
\]
\end{prop}
\begin{lem}[{\cite[Lemma 1]{SuppersymmetryAndLocalization}}]
\label{lem:sigma_function}Under the assumptions of the above proposition,
there exists an odd function $\sigma$ such that the zero set of $Q\sigma_{|M_{red}}$
coincides with $R_{Q}$ and $Q^{2}\sigma=0$.
\end{lem}
Before proving this lemma, we note that this function will be central
to most of this section so the importance of this lemma should not
be overlooked. 
\begin{proof}
Since $Q^{2}$ comes from an action of some compact torus $T$, and
induces an action on $\alpha_{M}$, there exists a $Q^{2}$ invariant
inner product $g$ on $\alpha_{M}$. Let us denote the induced isomorphism
between $\alpha_{M}$ and $\alpha_{M}^{*}$ also by $g$. There exists
an odd function $\sigma'$ on $M$ such that $\alpha^{*}(\sigma')=g(\alpha(Q))$.
There are numerous ways of seeing it, and perhaps the most direct
way is to note that the sheaf of sections of $\alpha_{M}^{*}$ is
naturally isomorphic to $J/J^{2}$ where $J$ is the sheaf of ideals
generated by odd functions. We can now average out by the action of
$T$ to get a $T$ invariant odd function $\sigma$. In particular
$Q^{2}\sigma=0$.

Since $Q$ and $g$ are $T$ invariant, we conclude that so is $g(\alpha(Q))$
and thus in particular $\alpha^{*}(\sigma)=\alpha^{*}(\sigma')$.
We now take a look at how $\sigma$ looks in local coordinates. If
$\{x_{1},\dots,x_{n},\theta_{1},\dots,\theta_{m}\}$ are coordinates
in some neighborhood such that 
\[
Q=(b^{\alpha}(x)+\dots)\p{\theta^{\alpha}}+(\dots)\p{x^{i}},
\]
then 
\[
\sigma=g_{\alpha\beta}b^{\alpha}\theta^{\beta}+\dots
\]
and 
\[
Q\sigma=g_{\alpha\beta}b^{\alpha}b^{\beta}+\dots.
\]
In particular, $Q\sigma_{|M_{red}}=\norm{\alpha(Q)}_{g}$ which shows
that the vanishing set of $Q\sigma_{|M_{red}}$ coincides with $R_{Q}$.
\end{proof}

\begin{proof}[Proof of \propref{general_localization_theorem}]
Define a function 
\[
\beta=\frac{\sigma}{Q\sigma}
\]
where $\sigma$ is the function from the lemma. The function $\beta$
is defined on the complement of $R_{Q}$ and satisfies $Q\beta=1$.

We first construct a convenient partition of unity of $M$. Considering
the action of $Q^{2}$ on $M_{red}$ we can find a $Q^{2}$ invariant
locally finite open cover $\{U_{i}\}$ and an open neighborhood $O\supset R_{Q}$
such that $O\subset U_{0}\subset U$ and $U_{i}\bigcap O=\emptyset$
for $i\neq0$. Taking any partition of unity $\{f_{i}\}$ subordinate
to $\{U_{i}\}$ and averaging out by the action of the torus $T$
inducing $Q^{2}$, we can assume that $\{f_{i}\}$ are $Q^{2}$ invariant.
Consider now the functions $g_{i}$ given by 
\[
g_{i}=\begin{cases}
Q(\beta f_{i}) & \mbox{ for }i>0\\
1-\sum_{i\neq0}g_{i} & \mbox{ for }i=0.
\end{cases}
\]
It is now not difficult to see that $g_{0}$ satisfies the requirements
of the proposition. Moreover, note that for $i\neq0$, we have 
\[
\int_{M}\mu g_{i}=\int_{M}\mathcal{L}_{Q}(\mu\beta f_{i})=0
\]
where the last equality follows from invariance of integration under
diffeomorphisms. We now have 
\[
\int_{M}\mu=\sum_{i}\int_{M}\mu g_{i}=\int_{M}\mu g_{0}.
\]
It is left to the reader to show that the above localization statement
holds for any $g$ as in the statement of the proposition. 
\end{proof}

\subsection{Localization Formula}

The previous section tells us when an integral of a density $\mu$
localizes to the vanishing set of an odd vector field $Q$. In this
section we will calculate this integral in terms of the local data
of $\mu$ and $Q$ near this vanishing space in the case the vanishing
space of $Q$ is a non-degenerate subsupermanifold.
\begin{defn}
\label{def:orientation_of_an_automorphism}Let $V$ be a vector space
over $\mathbb{R}$ and $E$ an automorphism of $V$ such that the
closure of the group generated by $e^{tE}$ in $GL(V)$ is compact.
Let $g$ be any inner product such that $e^{tE}\subset O(g)$. Define
\[
\ori(E)\in\ori(V)
\]
to be the orientation determined by the symplectic form $(x,y)\mapsto\langle x,Ey\rangle_{g}$.
\end{defn}
Note that the above construction is well defined since the space of
inner products $g$ satisfying $e^{tE}\subset O(g)$ is a convex set.
This can be seen from the fact that the condition $e^{tE}\subset O(g)$
is equivalent to $E$ being skew-adjoint with respect to $g$ which
is a linear condition on $g$. 

We recall that a symplectic form $w\in\bigwedge^{2}E^{*}$ on a $2n$
dimensional vector space defines an orientation via the volume form
$w^{n}\in\bigwedge^{2n}E^{*}$. Equivalently, a basis $\beta$ belongs
to the orientation defined by $\omega$ if $\sgn(\pff(\omega_{\beta}))=1$
where $\omega_{\beta}$ is the matrix of $\omega$ in the basis $\beta$.
Combining this with the definition of $\ori(E)$, we see that a basis
$\beta$ belongs to $\ori(E)$ if and only if for any positive definite
symmetric matrix $g$ such that $gE_{\beta}$ is skew-symmetric, 
\[
\sgn(\pff(gE_{\beta}))=1
\]
where $E_{\beta}$ is the matrix of $E$ in the basis $\beta$.
\begin{thm}
\label{thm:localization_formula} Let $M$ be a supermanifold, $Q$
and odd vector fields such that $Q^{2}$ is compact and $\mu$ a $Q$-invariant
compactly supported density. Suppose that the vanishing space of $Q$
is a non-degenerate subsupermanifold $N$. Let $\nu$ be the normal
bundle of $N$ and $L:=\cL_{Q}|_{\nu}$, which is an odd automorphism.
Then ${\rm codim}(N)=2l|2l$ for some $l\in\N$, $L^{2}$ restricted
to $\nu_{1}|_{N_{\red}}$ defines an element $o\in\ori_{(0,1)}(\nu)$
as in definition \ref{def:orientation_of_an_automorphism} and the
integral of $\mu$ over $M$ is given by 
\[
\int_{M}\mu=\left(-2\pi\right)^{l}\int_{N}\inner{\sqrt{\ber(L)}\otimes o,\mu}
\]
where $\sqrt{\ber(L)}\in(\ber(\nu)\otimes\ori_{(1,1)}(\nu)$ is constructed
in \subref{Berezinian-of-an-odd-isomorphism},\subref{Square-Root-of-1-d-modules}.
\end{thm}
This extends the result of \cite{SuppersymmetryAndLocalization} to
the case where $N$ is not an isolated point. 
\begin{proof}
The main tool for proving this theorem will be the function $\sigma$
constructed in \lemref{sigma_function}. Recall that $\sigma$ is
an odd function such that $Q^{2}\sigma=0$. Define 
\[
Z(\lambda):=\int_{M}\mu e^{i\lambda Q\sigma}.
\]
We then have 
\[
\frac{d}{d\lambda}Z(\lambda)=i\lambda\int_{M}\mu Q\sigma e^{i\lambda Q\sigma}=i\lambda\int_{M}\mathcal{L}_{Q}(\mu\sigma e^{i\lambda Q\sigma})=0.
\]
In particular $Z(\lambda)$ does not depend on $\lambda$. We can
thus calculate $\int_{M}\mu=Z(0)$ via stationary phase approximation
as $\l\ra\infty$ from previous section once we show that the critical
space of $Q\sigma$ is $N$ and calculate its Hessian.

We first show that the critical subsupermanifold of $Q\sigma$ contains
$N$. For $V$ a vector field tangent to $N$, we have 
\[
VQ\sigma|_{N}=0
\]
since $Q\sigma$ is constant on $N$. For any vector field $V$ on
$M$, we have 
\[
\mathcal{L}_{Q}(V)Q\sigma=\mathcal{L}_{Q}(VQ\sigma)-VQ^{2}\sigma=\mathcal{L}_{Q}(VQ\sigma)
\]
and since $Q|_{N}=0$ we have 
\[
\mathcal{L}_{Q}(V)Q\sigma|_{N}=0.
\]
Since $N$ is non-degenerate, the map $\mathcal{L}_{Q}$ restricted
to $\nu_{N}$ is an automorphism and thus in particular any section
of $TM_{|N}$ can be written as a sum of a restriction of a vector
field tangent to $N$ and an image of $\mathcal{L}_{Q}$. This shows
that $d(Q\sigma)$ vanishes on $N$.

The even codimension of $N$ equals the odd codimension of $N$ because
$\mathcal{L}_{Q}|_{\nu}$ is an odd automorphism. The odd codimension
of $N$ is even because $\mathcal{L}_{Q^{2}}|_{\nu_{1}}$ exponentiates
to an action of a torus.

We now need to calculate $H:=\hess(Q\sigma)$. For the time being,
when referring to $H$ we will mean the restriction of $H$ to $N$
(So far, we only showed that the critical subsupermanifold of $Q\sigma$
contains $N$, not that it equals $N$). Consider any two vector fields
$\xi_{1},\xi_{2}$ on $M$. We have 
\[
\xi_{1}\xi_{2}Q^{2}\sigma=0
\]
and thus, applying the Leibniz rule to $\mathcal{L}_{Q}$, we get
that 
\[
0=\left(-1\right)^{p(\xi_{1})+p(\xi_{2})}\xi_{1}\xi_{2}Q^{2}\sigma=\mathcal{L}_{Q}\left(\xi_{1}\xi_{2}Q\sigma\right)-\mathcal{L}_{Q}(\xi_{1})\xi_{2}Q\sigma-(-1)^{p(\xi_{1})}\xi_{1}\mathcal{L}_{Q}\left(\xi_{2}\right)Q\sigma
\]
and in particular, restricting to $N$ we get 
\[
\mathcal{L}_{Q}(\xi_{1})\xi_{2}Q\sigma+(-1)^{p(\xi_{1})}\xi_{1}\mathcal{L}_{Q}(\xi_{2})Q\sigma=0.
\]
In terms of the $H$ and $L$, this gives us 
\[
H(L\xi_{1},\xi_{2})+(-1)^{p(\xi_{1})}H(\xi_{1},L\xi_{2})=0.
\]
In terms of maps, the above equality states that 
\begin{equation}
\hat{H}\circ L=-L^{*}\circ\hat{H}\label{eq:HL=00003DLH}
\end{equation}
where 
\[
\hat{H}:\nu\rightarrow\nu^{*}
\]
is the map induced by $H$. The maps in \eqref{HL=00003DLH} are odd
maps from $\nu$ to $\nu^{*}$. In terms of matrices with respect
to some basis $\beta$ of $\nu$ the above equality is expressed by
\begin{equation}
\hat{H}L=-L{}^{st}\hat{H}.\label{eq:HL=00003DLHMatrices}
\end{equation}
If in block form, 
\begin{eqnarray*}
\hat{H} & = & \left(\begin{matrix}A & B\\
C & D
\end{matrix}\right)\\
L & = & \left(\begin{matrix}U & V\\
W & X
\end{matrix}\right),
\end{eqnarray*}
then the above equality implies that modulo odd variables, 
\begin{equation}
AV=W^{t}D.\label{eq:AV=00003DWD}
\end{equation}
By construction of $\sigma$, one can see that $A$ is non-degenerate
while $V$ and $W$ are non-degenerate by the fact that $L$ is an
automorphism. Thus in particular $D$ is non-degenerate and $N$ is
the critical subsupermanifold of $Q\sigma$. 

Defining the matrix 
\[
I'=\left(\begin{matrix}0 & \id\\
\id & 0
\end{matrix}\right)
\]
as in the coordinate description of the Berezinian of an odd automorphism,
it is not hard to see that 
\begin{eqnarray*}
\ber(L{}^{st}I') & = & \ber(LI')^{-1}\\
\ber(I'\hat{H}I') & = & \ber(\hat{H})^{-1}.
\end{eqnarray*}
In particular, it follows from \eqref{HL=00003DLHMatrices} that 
\[
\ber(\hat{H}LI')^{2}=1.
\]
We now show that $\ber(\hat{H}LI')=1$ modulo odd variables. Let $J$
be the ideal of $\mathcal{O}_{N}$ generated by odd functions. We
have 
\[
\hat{H}LI'\cong\left(\begin{matrix}AV & 0\\
0 & DW
\end{matrix}\right)\ \ \mbox{mod}(J).
\]
It thus follows from \eqref{AV=00003DWD} and the fact that the dimension
of $\nu_{1}$ is even that 
\[
\ber(\hat{H}LI')=1\ \ \mbox{mod}(J).
\]
We now show that in fact $\ber(\hat{H}LI')=1$. Let $\alpha=\ber(\hat{H}LI')-1$.
We have that $\alpha\in J^{2}$ and satisfies 
\[
2\alpha+\alpha^{2}=0
\]
since $(1+\alpha)^{2}=1$. It follows from the equation above that
if $\alpha\in J^{k}$ then also $\alpha\in J^{2k}$. Since for sufficiently
large $k$ we have that $J^{k}=0$, we conclude that $\alpha=0$. 

Recalling the definition of $\ber(L)$ and $\ber(H)$, the equality
\[
\ber(\hat{H}LI')=1
\]
implies that 
\[
\ber(L)=\ber^{-1}(H).
\]

We now show that the $(0,1)$ orientation defined by $H$ differs
from the one defined by $L^{2}$ by $(-1)^{l}$. For the remaining
of this paragraph, all of the quantities considered will be taken
restricted to $N_{red}$ or equivalently modulo the ideal $J$ generated
by odd functions. Denote by $H_{\beta}$ the matrix $H(e_{i},e_{j})$.
We have $H_{\beta}=\hat{H}{}^{st}$ and therefore the $(0,1)$ orientation
defined by $H$ is given by the sign of $\pff(D^{t})$. The orientation
defined by $L^{2}$ is given by the automorphism of $\nu_{1}$ which
in the fixed basis is given by the matrix $WV$. The desired relationship
between orientations will follow from \eqref{AV=00003DWD}
\[
AV=W^{t}D.
\]
Define 
\[
g:=W^{-t}AW^{-1}.
\]
From the definition of $\sigma$, it follows that $A$ is positive
definite, and therefore so is $g$. We then have 
\[
D=gWV
\]
and therefore orientation defined by $WV$ is determined by 
\[
\pff(gWV)=\pff(D)=(-1)^{l}\pff(D^{t}).
\]
This shows that $\ori_{(0,1)}(H)=(-1)^{l}o\in\ori_{(0,1)}(\nu)$ where
$o$ is the orientation of $\nu_{1}$ defined by $L^{2}|_{N_{red}}$.

The theorem now follows by applying stationary phase approximation
(\thmref{super_stationary_phase_approx}) to $Z(\lambda)$, noting
that $Q\sigma$ vanishes on $N$ and that $\sgn(H_{\red})=2l$.

\end{proof}

\newcommand{\etalchar}[1]{$^{#1}$}

\end{document}